\documentclass[12pt,a4paper]{amsart}
\usepackage{amsfonts}
\usepackage{amsthm}
\usepackage{amsmath}
\usepackage{amssymb}
\usepackage{amscd}
\usepackage{t1enc}
\usepackage[margin=2.9cm]{geometry}
\usepackage{epstopdf} 
\usepackage{times}
\usepackage{enumerate}
\usepackage{changepage}
\usepackage{mfirstuc}  
\usepackage{tikz}
\usepackage{cite}
\usepackage{booktabs}
\usepackage{tabularx}
\numberwithin{equation}{section}

\theoremstyle{plain}
\newtheorem{Th}{Theorem}[section]
\newtheorem{Lemma}[Th]{Lemma}

\newtheorem{Cor}[Th]{Corollary}
\newtheorem{Prop}[Th]{Proposition}
\newtheorem{Def}[Th]{Definition}

\newcommand{\Q}{\mathbb{Q}}
\newcommand{\R}{\mathbb{R}}
\newcommand{\Z}{\mathbb{Z}}
\newcommand{\F}{\mathbb{F}}

\newcommand{\LF}{\mathrm{L}}

\newcommand{\ord}{\operatorname{ord}}
\newcommand{\Vol}{\operatorname{Vol}}
\newcommand{\Tr}{\operatorname{Tr}}

\begin{document}

\title{L-functions of certain exponential sums over finite fields II}

\author{Xin Lin}
\address{Department of Mathematics, Shanghai Maritime University, Shanghai 201306, PR China}
\email{xlin1126@hotmail.com}

\author{Chao Chen}
\address{Department of Mathematics, University of California, Irvine, 
Irvine, CA 92697-3875, USA}
\email{chaoc12@uci.edu}



\begin{abstract}
In this paper, we compute the $q$-adic slopes of the L-functions of an important class of exponential sums arising from analytic number theory. Our main tools include Adolphson-Sperber's work on toric exponential sums and Wan's decomposition theorems. 
\end{abstract}

\maketitle

\section{Introduction}\label{sec1}
A fundamental problem in number theory is to estimate the reciprocal zeros and poles of L-functions of certain exponential sums over finite fields of characteristic $p$. Deligne's theorem on Riemann hypothesis gives general information for the complex absolute values of the zeros and poles. The theorem also indicates that the reciprocal roots and poles are $\ell$-adic units if $\ell$ is a prime and $\ell \neq p$. So the remaining unknown part is the $p$-adic estimate of the L-function which can be identified with the $p$-adic version of Riemann hypothesis for the L-function. In this paper, we study the $p$-adic slopes of L-functions of a special family of exponential sums.

Let $\F_{q}$ be the finite field of $q$ elements with characteristic $p$. For each positive integer $k$, let $\F_{q^k}$ denote the degree $k$ finite extension of $\F_{q}$ and $\F_{q^k}^*$ denote the set of non-zero elements in $\F_{q^k}$. Let $\psi: \F_{p} \rightarrow \mathbb{C}^*$ be a fixed nontrivial additive character over $\F_{p}$ and $\Tr_k: \F_{q^k} \rightarrow \F_{p}$ be the trace map. Let $n=\sum^{r}_{i=1} n_i$ be a partition of a positive integer $n$. We introduce $n$ variables 
$$x_{ij},\quad \text{for}\  1\leq i \leq r,\ 1\leq j\leq n_i,$$
and $n$ positive integral constants
$$b_{ij}, \quad \text{for}\  1\leq i \leq r,\ 1\leq j\leq n_i.$$ 

In this paper, we are concerned with the following class of exponential sums: for $a_i\in \F_{q}^*$ and $b_{ij}\in\Z_{>0}$,
\begin{equation*}
S_k(\vec{a})=\sum_{{\sum^{r}_{i=1}\frac{a_i}{\prod_{j=1}^{n_i}x_{ij}^{b_{ij}}}=1}}\psi\left(\Tr_k\left(\sum_{i=1}^r\sum^{n_i}_{j=1}x_{ij}\right)\right),
\end{equation*}
where the sum is over all $x_{ij} \in \F^*_{q^k}$. Since the case of $ b_{ij}$ divisible by $p$ can be reduced to $p \nmid b_{ij}$ through a change of variables, in this paper we assume $p \nmid b_{ij}$ for $1\leq i \leq r,\ 1\leq j\leq n_i$. 

The estimate of $S_k(\vec{a})$ plays an important role in analytic number theory when $n=4$, $n_i=2$ and $b_{ij}=1$. It is a crucial ingredient in Heath Brown's work on the divisor function $d_3(n)$ in arithmetic progressions\cite{heath1986divisor}. It also appears in Friedlander and Iwaniec's work on estimating certain averages of incomplete Kloosterman sums in application to the divisor problem of $d_3(n)$\cite{friedlander1985incomplete}. Relying on the estimate of $S_k(\vec{a})$ and Friedlander-Iwaniec's result, Zhang gained a boundary of the error terms in his work on twin prime conjecture\cite{zhang2014bounded}. 

For the complex absolute values of exponential sums, Katz \cite{MR887396} proved that 
\begin{equation*}
|S_k(\vec{a})|\leq c_1q^{(n-1)k/2},
\end{equation*}
where $c_1=\prod^{r}_{i=1}\left(1+\sum_{j=1}^{n_i}b_{ij}\right)-1$.

To understand the $p$-adic absolute values of the exponential sums, we study $p$-adic slopes of the reciprocal roots and poles of the associated L-function
\begin{align}\label{eq11}
\LF(\vec{a},T)=\exp  \left(\sum^\infty _ {k=1} S_k(\vec{a}) \frac{T^k}{k}\right).
\end{align}
When $n=4$, $n_i=2$ and $b_{ij}=1$, our previous paper\cite{chen2020lfunctions} showed that
\begin{equation*}
	\LF(\vec{a},T)=(1-T)(1-qT)\prod_{i=1}^6(1-\alpha_iT),
	\end{equation*}
where the $p$-adic norms of the reciprocal roots are given by
\begin{align*}
		|\alpha_i|_p=
		\begin{cases}
			1 &\text{if}\quad i =1.\\
			q^{-1} &\text{if}\quad i= 2,3.\\
			q^{-2} &\text{if}\quad i= 4,5.\\
			q^{-3} &\text{if}\quad i= 6.
		\end{cases}
	\end{align*}
In this paper, we obtain $q$-adic slopes for the above general class of exponential sums. Note that our main result is consistent with Katz's complex estimate.
\begin{Th}\label{thm1} Let $D= \mathrm{lcm}\{b_{ij}| 1\leq i \leq r, 1\leq j \leq n_i\}$. 
\begin{enumerate}
\item The associated L-function is a polynomial  given by
\begin{align*}	
	\LF(\vec{a},T)^{(-1)^n}&=(1-T)^{r-1}\prod^{d-r}_{i=1}(1-\alpha_iT),
\end{align*}		
	where  $d=\prod^{r}_{i=1}\left(1+\sum_{j=1}^{n_i}b_{ij}\right)$.
\item If $p \equiv 1 (\bmod D)$, for each $m=0,1,\ldots, nD$, the number of reciprocal zeros of $\LF(\vec{a},T)^{(-1)^{n}}$ with $q$-adic slope $m/D$ is the coefficient of $x^{m+D}$ in the following generating function 
\begin{align*}
		G(x)=&\left(1-x^D\right)^{n-r}\prod^{r}_{i=1}\frac{1-x^{(1+\sum^{n_i}_{j=1}\frac{1}{b_{ij}})D}}{\prod_{j=1}^{n_i}\left(1-x^{D/b_{ij}}\right)},
		\end{align*}
and for any rational number $m\notin \{0,1,\ldots, nD\}$, there is no reciprocal zero of $\LF(\vec{a},T)^{(-1)^{n}}$ with $q$-adic slope $m/D$. 
\end{enumerate}	
	
\end{Th}

\begin{Cor}
Assume $b_{ij}=1$ ($1\leq i \leq r, 1 \leq j \leq n_i$).
\begin{enumerate}
\item The polynomial $\LF(\vec{a},T)^{(-1)^{n}}$ has degree $\prod^{r}_{i=1}\left(1+n_i\right)-1$.
\item Let $d=\prod^{r}_{i=1}\left(1+n_i\right)$. If $n_i$ is even for $1\leq i\leq n$, we have
	\begin{align*}	
			\LF(\vec{a},T)^{(-1)^n}&= (1-T)^{r-1}(1-\gamma_0T)(1-qT)^{\frac{r^2-r}{2}}\prod^{r}_{i=1}(1-\gamma_iT)\prod^{d-\frac{r^2+3r+2}{2}}_{j=1}(1-\alpha_jT),
			\end{align*}
			where $\ord_q\gamma_0=0$, $\ord_q\gamma_i=1$ $(1\leq i \leq r)$, $|\gamma_i|=q^{\frac{n-1}{2}}$ $(0\leq i \leq r)$, $\ord_q\alpha_j > 1$ and $|\alpha_j |\leq q^{\frac{n-1}{2}}$.
\item For each $m=0,1,2,\ldots, n$, the number of reciprocal zeros of $\LF(\vec{a},T)^{(-1)^{n}}$ with $q$-adic slope $m$ is the coefficient of $x^{m+1}$ in the following generating function
	\begin{align*}
	G_1(x)=\prod_{i=1}^{r}\left(1+x+\cdots+x^{n_i}\right),
	\end{align*}
and for any rational number $m\notin \{0,1,\ldots, n\}$, there is no reciprocal zero of $\LF(\vec{a},T)^{(-1)^{n}}$ with $q$-adic slope $m$.
\end{enumerate}	

\end{Cor}

Our approach is to reduce this theorem to the L-function of toric exponential sums 
and then apply the systematic results available for such toric L-functions. Let $N_i=\sum^{i}_{l=1}n_l$ for $1\leq i \leq r$ and $N_0=0$. Let $x_{ij}=x_{N_{i-1}+j}$ and $b_{ij}=b_{N_{i-1}+j}$ ($1\leq i \leq r$, $1\leq j \leq n_i$). Consider a Laurent polynomial $f\in \F_{q}[x_1^{\pm1},\ldots, x_{n+1}^{\pm1}]$ defined by
\begin{equation*}
f(x_1,x_2,\ldots,x_{n+1})=\sum_{i=1}^{n}x_i+x_{n+1}\left({\sum^{r}_{i=1}\frac{a_i}{\prod_{l=N_{i-1}+1}^{N_i}x_{l}^{b_{l}}}}-1\right),
\end{equation*}
where $a_i\in \F_{q}^*$ and $b_{l}\in\Z_{>0}$. For any positive integer $k$, let 
 \begin{equation*}
 S^*_k(f)=\sum_{x_i\in \F^*_{q^k}}\psi(\Tr_k(f))
 \end{equation*}
 be the associated exponential sum. Its generating L-function is defined to be
 \begin{equation*}
 \LF^*(f,T)=\exp  \left(\sum^\infty _ {k=1} S^*_k(f) \frac{T^k}{k}\right).
 \end{equation*}

The following equation describes the relationship between $S_k(\vec{a})$ and  $S^*_k(f)$,
 \begin{equation*}
 S_k(\vec{a})=\frac{(-1)^n}{q^k}+\frac{1}{q^k}S^*_k(f).
 \end{equation*}
Based on the relationship, it's easy to check that
\begin{align}\label{eq12}
\LF(\vec{a},T)^{(-1)^n}=\frac{1}{1-T/q}\LF^*\left(f,T/q\right)^{(-1)^n}.
\end{align}
So it suffices to evaluate the reciprocal roots or poles of $\LF^*(f,T)$. By Adolphson and Sperber's theorem\cite{AS1989},  $\LF^*(f,T)^{(-1)^{n}}$ is a polynomial if $p\nmid \mathrm{lcm}\{b_{ij}\}$. Under this assumption, we compute the $q$-adic Newton polygon of $\LF^*(f,T)^{(-1)^{n}}$, which indicates the $q$-adic slope information of the L-function. Furthermore, using Wan's decomposition theorem, we show that the L-function has the following form
\begin{align*}	
	\LF^*(f,T)^{(-1)^n}&=(1-T)(1-qT)^{r-1}\prod^{d-r}_{i=1}(1-\beta_iT),
\end{align*}
	where  $d=\prod^{r}_{i=1}\left(1+\sum_{j=1}^{n_i}b_{ij}\right)$. This leads to our main theorem.

This paper is organized as follows. In section \ref{sec2}, we review some technical methods and theorems including Adolphson-Sperber's theorems and Wan's decomposition theorems. In section \ref{sec3}, we prove the main results.

\section{Preliminaries}\label{sec2}
\subsection{Rationality of the generating L-function}\label{subsec21}
Let $f\in\F_{q}\left[x_1^{\pm1}, \ldots, x_n^{\pm1} \right]$ be a Laurent polynomial and its associated exponential sum is defined to be 
\begin{align}\label{eq1}
S^*_k(f)=\sum_{x_i \in \F^*_{q^k}} \psi(\Tr_k(f)),
\end{align}
where $\Tr_k: \F_{q^k} \rightarrow \F_{p}$ is the trace map and $\psi: \F_{p} \rightarrow \mathbb{C}^*$ is a fixed nontrivial additive character. 
It's a classical problem to give a good estimate for the valuations of $S^*_k(f)$ in analytic number theory. In order to compute the absolute values of the exponential sums, we usually study the generating L-function of $S^*_k(f)$ given by
\begin{equation*}
\LF^*(f,T)=\exp  \left(\sum^\infty _ {k=1} S^*_k (f) \frac{T^k}{k} \right) \in \Q(\zeta_p)[[T]].
\end{equation*}
By a theorem of Dwork-Bombieri-Grothendieck\cite{Dwork1962OnTZ, Grothendieck1964}, the generating L-function is a rational function,
\begin{equation*}
\LF^*(f,T)=\frac{\prod^{d_1}_{i=1}(1-\alpha_iT)}{\prod^{d_2}_{j=1}(1-\beta_jT)},
\end{equation*}
where all the reciprocal zeros and poles are non-zero algebraic integers. After taking logarithmic derivatives, we have the formula
\begin{align}\label{expo}
S_k^*(f)=\sum^{d_2}_{j=1}\beta_j^k-\sum^{d_1}_{i=1}\alpha_i^k, \quad k \in \Z_{\geq1},
\end{align}
which implies that the zeros and poles of the generating L-function contain critical information about the exponential sums.

From Deligne's theorem on Riemann hypothesis \cite{Deligne1980}, the complex absolute values of reciprocal zeros and poles are bounded as follows
\begin{equation*}
|\alpha_i|=q^{u_i/2}, |\beta_j|=q^{v_j/2}, u_i\in \Z \cap [0,2n], v_j\in \Z \cap [0,2n].
\end{equation*}
For non-archimedean absolute values, Deligne\cite{Deligne1980} proved that $|\alpha_i|_\ell=|\beta_j|_\ell=1$ when $\ell$ is a prime and $\ell\neq p$. Depending on Deligne's integrality theorem, we have the following estimates for $p$-adic absolute values
\begin{equation*}
|\alpha_i|_p=q^{-r_i}, |\beta_j|_p=q^{-s_j}, r_i\in \Q \cap [0,n], s_j\in \Q \cap [0,n].
\end{equation*}
The integer $u_i$ (resp. $v_j$) is called the \textit{weight} of $\alpha_i$ (resp. $\beta_j$) and the rational number $r_i$
 (resp. $s_j$) is called the \textit{slope} of $\alpha_i$ (resp. $\beta_j$). In the past few decades, it has been tremendous interest in determining the weights and slopes of the generating L-functions. Without any further condition on the Laurent polynomial $f$ or prime $p$, it's even hard to determine the number of reciprocal roots and poles. Adolphson and Sperber\cite{AS1989} proved that under a suitable smoothness condition of $f$ in $n$ variables, the associated L-function $\LF^*(f,T)^{(-1)^{n-1}}$ is a polynomial and one can determine the slopes of the reciprocal roots using Newton polygons. Adolphson and Sperber\cite{AS1989} also proved that if $\LF^*(f,T)^{(-1)^{n-1}}$ is a polynomial, its Newton polygon has a lower bound called Hodge polygon. The basic definitions of the Newton polygon and the Hodge polygon will be discussed in the next subsection.

\subsection{Newton polygon and Hodge polygon}\label{subsec22}
Let 
\begin{equation*}
f(x_1, \ldots x_n)=\sum^J_{j=1} a_j x^{V_j}
\end{equation*}
 be a Laurent polynomial with $a_j \in \F ^*_{q}$ and $V_j=(v_{1j},\ldots, v_{nj})\in \Z^n$ $(1\leq j\leq J)$. The \emph{Newton polyhedron} of $f$, $\Delta(f)$, is defined to be the convex closure in $\R^n$ generated by the origin and the lattice points $V_j$ ($1\leq j \leq J$). For $\delta \subset \Delta(f)$, let the Laurent polynomial
\begin{equation*}
f^{\delta}=\sum_{V_j\in \delta}a_j x^{V_j}
\end{equation*}
be the restriction of $f$ to $\delta$. 
 \begin{Def}
 	A Laurent polynomial $f$ is called non-degenerate if for each closed face $\delta$ of $\Delta(f)$ of arbitrary dimension which doesn't contain the origin, the $n$-th partial derivatives 
 	\begin{equation*}
 	\left\{ \frac{\partial f^{\delta}}{\partial x_1 }, \ldots, \frac{\partial f^{\delta}}{\partial x_n} \right\}
 	\end{equation*} 
 	have no common zeros with $x_1\ldots x_n \neq 0$ over the algebraic closure of $\F_q$.
 \end{Def}
\begin{Th}[Adolphson and Sperber\cite{AS1989}]\label{thm2}
	For any non-degenerate $f\in  \F_{q}[x_1^{\pm1},\ldots, x_n^{\pm1}]$, the associated L-function $\LF^*(f,T)^{(-1)^{n-1}}$ is of the following form, 
	\begin{equation*}
	\LF^*(f,T)^{(-1)^{n-1}}=\prod^{n! \Vol(\Delta(f))}_{i=1}(1-\alpha_i T),
	\end{equation*}
	where $|\alpha_i|=q^{\omega_i/2}$, $\omega_i\in \Z\cap [0,n]$ and $i=1,2,\ldots, n!\Vol(\Delta).$
\end{Th}
Deligne's integrality theorem implies that the $p$-adic absolute values of reciprocal roots are given by $|\alpha_i|_p=q^{-r_i}$ where $r_i\in \Q\cap[0,n]$. For simplicity, we normalize $p$-adic absolute value to be $|q|_p=q^{-1}$. To determine the $q$-adic slopes of its reciprocal roots, we compute the $q$-adic Newton polygon of $\LF^*(f,T)^{(-1)^{n-1}}$.
\begin{Def}[Newton polygon]\label{m1}
	Let $\LF(T)=\sum^n_{i=0}a_i T^i$ $\in 1+T\overline{\Q}_p[T]$, where $\overline{\Q}_p$ is the algebraic closure of $\Q_p$. The $q$-adic Newton polygon of $\LF(T)$ is defined to be the lower convex closure of the set of points $\{\left(k,\ord_q(a_k)\right) | k=0, 1,\ldots, n \}$ in $\R^2$.
\end{Def}
Here $\ord_q$ denotes the standard $q$-adic ordinal on $\overline{\Q}_p$ where the valuation is normalized by assuming $\ord_q(q)=1$. The following lemma\cite{koblitz2012p} relates the $q$-adic valuation of reciprocal roots to the shape of the corresponding $q$-adic Newton polygon.
\begin{Lemma}\label{le4}
	In the above notation, let $\LF(T)=(1-{\alpha_1}T)\ldots(1-{\alpha_n}T)$ be the factorization of $\LF(T)$ in terms of reciprocal roots $\alpha_i \in \overline{\Q}_p$. Let $\lambda_i=\ord_q\alpha_i$. If $\lambda$ is the slope of a $q$-adic Newton polygon with horizontal length $l$, then precisely $l$ of the $\lambda_i$ are equal to $\lambda$.
\end{Lemma}
Assume $f$ is a non-degenerate Laurent polynomial in $n$ variables and consequently $\LF^*(f,T)^{(-1)^{n-1}}$ is a polynomial. The $q$-adic Newton polygon of $\LF^*(f,T)^{(-1)^{n-1}}$ is denoted by NP($f$), which is hard to compute directly. Adolphson and Sperber proved that NP($f$) has a topological lower bound called the Hodge polygon, that is easier to calculate. So generally, we first compute its lower bound Hodge polygon and then determine when the Newton polygon coincides with its lower bound. 

Let $\Delta$ be an $n$-dimensional integral polytope containing the origin in $\R^n$. Define $C(\Delta)$ to be the cone generated by $\Delta$ in $\R^{n}$. For any point $u\in \R^{n}$, the weight function $w(u)$ is the smallest non-negative real number $c$ such that $u\in c\Delta$. Let $w(u)=\infty$ if such $c$ doesn't exist. Assume $\delta$ is a co-dimension 1 face of $\Delta$ not containing the origin. Let $D(\delta)$ be the least common multiple of the denominators of the coefficients in the implicit equation of $\delta$, normalized to have constant term 1.
We define the denominator of $\Delta$ to be the least common multiple of all such $D(\delta)$ given by:
\begin{equation*}
D=D(\Delta)= \mathrm{lcm}_{\delta}D(\delta)
\end{equation*}
where $\delta$ runs over all the co-dimension 1 faces of $\Delta$ that don't contain the origin. It's easy to check 
\begin{equation*}
w(\Z^n)\subseteq \frac{1}{D(\Delta)}\Z_{\geq0}\cup \{ + {\infty} \}.
\end{equation*} 
For a non-negative integer $k$, let 
 \begin{equation}\label{eqW}
 W_{\Delta}(k)=\# \left\{ u \in \Z^n | w(u)= \frac{k}{D} \right\}
 \end{equation} 
 be the number of lattice points in $\Z^n$ with weight $k/D$. 

\begin{Def}[Hodge number]\label{def5}
	Let $\Delta$ be an $n$-dimensional integral polytope containing the origin in $\R^n$. For a non-negative integer $k$, the $k$-th Hodge number of $\Delta$ is defined to be
\begin{align}\label{eq3}
	H_{\Delta}(k)=\sum^{n}_{i=0}(-1)^i \binom{n}{i}W_{\Delta}(k-iD).
\end{align}
\end{Def}
It's easy to check that 
\begin{equation*}
H_{\Delta}(k)=0, \quad \text{if}\quad k>nD.
\end{equation*}
Adolphson and Sperber\cite{AS1989} proved that $H_{\Delta}(k)$ coincides with the usual Hodge number in the toric hypersurface case that $D=1$. Based on the Hodge numbers, we define the Hodge polygon of a given polyhedron $\Delta\in \R^n$ as follows.
\begin{Def}[Hodge polygon]\label{def6}
	The Hodge polygon HP($\Delta$) of $\Delta$ is the lower convex polygon in $\R^2$ with vertices (0,0) and 
	\begin{equation*}
	Q_k=\left( \sum^k_{m=0}H_{\Delta}(m), \frac{1}{D}\sum^k_{m=0}m H_{\Delta}(m) \right), \quad  k=0,1,\ldots, nD,
	\end{equation*}
	where $H_{\Delta}(k)$ is the $k$-th Hodge number of $\Delta$, $k=0,1,\ldots, nD.$
	
	That is, HP($\Delta$) is a polygon starting from origin (0,0) with a slope $k/D$ side of horizontal length $H_{\Delta}(k)$ for $k=0,1,\ldots, nD$. The vertex $Q_k$ is called a break point if $H_{\Delta}(k+1)\neq 0$ where $k=1,2,\ldots,nD-1$. 
\end{Def}
Here the horizontal length $H_{\Delta}(k)$ represents the number of  lattice points of weight $k/D$ in a certain fundamental domain corresponding to a basis of the $p$-adic cohomology space used to compute the L-function. Adolphson and Sperber constructed the Hodge polygon and proved that it's a lower bound of the corresponding Newton polygon. 
\begin{Th}[Adolphson and Sperber\cite{AS1989}]
	For every prime p and non-degenerate Laurent polynomial $f$ with $\Delta(f)=\Delta \subset \R^n$, we have 
	\begin{equation*}
	\text{NP}(f) \geq \text{HP}(\Delta),
	\end{equation*}
	where NP($f$) is the $q$-adic Newton polygon of  $\LF^*(f,T)^{(-1)^{n-1}}.$
	Furthermore, the endpoints of NP($f$) and NP($\Delta$) coincide.
\end{Th}
\begin{Def}
	A Laurent polynomial $f$ is called ordinary if NP($f$) = HP($\Delta$). 
\end{Def}

Apparently, the ordinary property of a Laurent polynomial depends on its Newton polyhedron $\Delta$. In order to study the ordinary property, we will apply Wan's decomposition theorem \cite{Dwan1993}, decomposing the polyhedron $\Delta$ into small pieces that are much easier to deal with. 

\subsection{Wan's decomposition theorems}\label{subsec23}
 \subsubsection{Facial decomposition theorem}\label{subsubsec231}
In this paper, we use facial decomposition theorem to cut the polyhedron into small simplices. For each simplex, we can apply some criteria to determine the non-degenerate and ordinary property. 
\begin{Th}[Facial decomposition theorem\cite{Dwan1993}]\label{thm9}
	Let $f$ be a non-degenerate Laurent polynomial over $\F_q$. Assume $\Delta=\Delta(f)$ is $n$-dimensional and $\delta_1,\ldots, \delta_h$ are all the co-dimension 1 faces of $\Delta$ which don't contain the origin. Let $f^{\delta_i}$ denote the restriction of $f$ to $\delta_i$. Then $f$ is ordinary if and only if $f^{\delta_i}$ is ordinary for $1\leq i\leq h$. 
\end{Th}
In order to describe the boundary decomposition, we first express the L-function in terms of the Fredholm determinant of an infinite Frobenius matrix.

\subsubsection{Dwork's trace formula}\label{subsubsec232}
Let $p$ be a prime and $q=p^a$ for some positive integer $a$. Let $\Q_p$ denote the field of $p$-adic numbers and $\Omega$ be the completion of $\overline{\Q}_p$. Pick a fixed primitive $p$-th root of unity in $\Omega$ denoted by $\zeta_p$. In $\Q_p(\zeta_p)$, choose a fixed element $\pi$ satisfying
\begin{equation*}
\sum^{\infty}_{m=0}\frac{\pi^{p^m}}{p^m}=0 \quad \text{and} \quad \ord_p \pi =\frac{1}{p-1}.
\end{equation*} 
By Krasner's lemma, it's easy to check $Q_p(\pi)=Q_p(\zeta_p)$. Let $K$ be the unramified extension of $\Q_p$ of degree $a$. Let $\Omega_a$ be the compositum of $Q_p(\zeta_p)$ and $K$. 
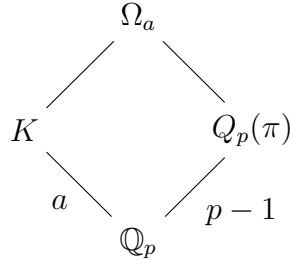
\begin{center}
	\begin{tikzpicture}[node distance = 1.5cm, auto]
      \node (Q) {$\Q_p$};
      \node (K) [above of=Q, left of=Q] {$K$};
      \node (Z) [above of=Q, right of=Q] {$Q_p(\pi)$};
      \node (O) [above of=Q, node distance = 3cm] {$\Omega_a$};
      \draw[-] (Q) to node {$a$} (K);
      \draw[-] (Q) to node [swap] {$p-1$} (Z);
      \draw[-] (K) to node {$$} (O);
      \draw[-] (Z) to node [swap] {$$} (O);
      \end{tikzpicture}
  \end{center}
Define the Frobenius automorphism $\tau \in \text{Gal}(\Omega_a/\Q_p(\pi))$ by lifting the Frobenius automorphism $x\mapsto x^p$ of Gal($\F_q/\F_p$) to a generator $\tau$ of Gal($K/\Q_p$) and extending it to $\Omega_a$ with $\tau(\pi)=\pi$. For the primitive $(q-1)$-th root of unity $\zeta_{q-1}$ in $\Omega_a$, we have $\tau(\zeta_{q-1})=\zeta_{q-1}^p$.

Let $E_p(t)$ be the Artin-Hasse exponential series,
\begin{equation*}
E_p(t)=\exp\left(\sum^{\infty}_{m=0}\frac{t^{p^m}}{p^m}\right)=\sum_{m=0}^{\infty}\lambda_m t^m \in  \Z_p[[x]].
\end{equation*}
In Dwork's terminology, a splitting function $\theta(t)$ is defined to be
\begin{equation*}
\theta(t)=E_p(\pi t)=\sum_{m=0}^{\infty}\lambda_m\pi^mt^m.
\end{equation*}
When $t=1$, $\theta(1)$ can be identified with $\zeta_p$ in $\Omega$.

Consider a Laurent polynomial $f \in $ $\F_q[x_1^{\pm1}, \ldots, x_n^{\pm1}]$ given by
\begin{equation*}
f=\sum_{j=1}^J \bar{a}_j x^{V_j},
\end{equation*}
where $V_j \in {\Z}^n$ and $\bar{a}_j \in \F_q^{*}$. Let $a_j$ be the Teichm\"{u}ller lifting of $\bar{a}_j $ in $\Omega$ satisfying $a_j^q=a_j$. Let
\begin{equation*}
F(f,x)=\prod_{j=1}^J\theta(a_j x^{V_j})=\sum_{r \in {\Z}^n} F_r(f)x^r  \in \Omega_a[[x]].
\end{equation*}
The coefficients of $F(f,x)$ are given by 
\begin{equation*}
F_r(f)=\sum_u (\prod^{J}_{j=1} \lambda_{u_j} a_j^{u_j}) \pi^{u_1+\dots+u_{J}}, \quad r \in {\Z}^n,
\end{equation*}
 where the sum is over all the solutions of the following linear system
\begin{equation*}
\sum^{J}_{j=1}u_jV_j=r \quad \text{with}\quad u_j \in \Z_{\geq 0},
\end{equation*} 
and $\lambda_m$ is $m$-th coefficient of the Artin-Hasse exponential series $E_p(t)$.

Assume $\Delta=\Delta(f)$. Let $L(\Delta)=\Z^{n}\cap C(\Delta)$ be the set of lattice points in the closed cone generated by origin and $\Delta$. For a given point $r\in \R^n$, define the weight function to be
\begin{equation*}
w(r): =\inf_{\vec{u}}\left\{ \sum_{j=1}^J u_j |\sum_{j=1}^J u_jV_j=r,\quad u_j\in \R_{\geq 0}\right\}.
\end{equation*}
In Dwork's terminology, the infinite semilinear Frobenius matrix $A_1(f)$ is a matrix whose rows and columns are indexed by the lattice points in $L(\Delta)$ with respect to the weights
\begin{equation*}
	A_1(f)=(a_{r,s}(f))=(F_{ps-r}(f)\pi^{w(r)-w(s)}),
\end{equation*}
where $r, s\in L(\Delta)$. 
Based on the fact that $\ord_p F_r(f)\geq \frac{w(r)}{p-1}$, we have the following estimate
\begin{equation*}
\ord_p( a_{r,s}(f)) \geq \frac{w(ps-r)+w(r)-w(s)}{p-1}\geq w(s).
\end{equation*} 
Let $\xi$ be an element in $\Omega$ satisfying $\xi^D=\pi^{p-1}$. Then $A_1(f)$ can be written in a block form,
\begin{equation*}
A_1(f)
=\begin{pmatrix}
A_{00} &  \xi A_{01}  & \cdots\quad & {\xi}^iA_{0i}&\cdots\\
A_{10} &  \xi A_{11}  & \cdots\quad & {\xi}^iA_{1i}&\cdots\\
 \vdots & \vdots & \ddots  & \vdots  \\
 A_{i0} &  \xi A_{i1}  & \cdots\quad & {\xi}^iA_{ii}&\cdots\\
 \vdots & \vdots & \ddots  & \vdots
\end{pmatrix},
\end{equation*}
where the block $A_{ii}$ is a $p$-adic integral $W_{\Delta}(i) \times W_{\Delta}(i)$ matrix and $W_{\Delta}(i)=\#\{u \in \Z^n | w(u)= \frac{i}{D}\}$.

The infinite linear Frobenius matrix $A_a(f)$ is defined to be
\begin{equation*}
A_a(f)=A_1(f)A_1^{\tau}(f)\cdots A_1^{\tau^{a-1}}(f).
\end{equation*}
The $q$-adic Newton polygon of $\det(I-TA_a(f))$ has a natural lower bound which can be identified with the chain level version of the Hodge polygon.
\begin{Def}Let $P(\Delta)$ be the polygon in $\R^2$ with vertices $(0,0)$ and 
\begin{equation*}
P_k=\left( \sum^k_{m=0}W_{\Delta}(m), \frac{1}{D}\sum^k_{m=0}m W_{\Delta}(m) \right), \quad  k=0,1,2, \ldots
\end{equation*}
\end{Def}
\begin{Prop}[\cite{Dwan2004}]The $q$-adic Newton polygon of $\det(I-TA_a(f))$ lies above $P(\Delta).$
\end{Prop}

By Dwork's trace formula, we can identify the associated L-function with a product of some powers of the Fredholm determinant\cite{Dwan2004},
\begin{equation}\label{eq24}
\LF^{*}(f,T)^{(-1)^{n-1}}=\prod_{i=0}^{n}\det(I-Tq^{i}A_a(f))^{(-1)^i\binom{n}{i}}.
\end{equation}
Equivalently, we have,
\begin{equation}\label{eq25}
\det(I-TA_a(f))=\prod_{i=0}^{\infty} \left( \LF^{*}(f,q^iT)^{(-1)^{n-1}} \right)^{\binom{n+i-1}{i}}.
\end{equation}
\begin{Prop}[\cite{Dwan2004}]\label{prop c}
Notations as above. Assume $f$ is non-degenerate with $\Delta=\Delta(f)$. Then $\mathrm{NP}(f)=\mathrm{HP}(\Delta)$ if and only if the $q$-adic Newton polygon of $\det(I-TA_a(f))$ coincides with its lower bound $P(\Delta).$
\end{Prop}

\subsubsection{Boundary decomposition}\label{subsubsec233}
Let $f \in $ $\F_q[x_1^{\pm1}, \ldots, x_n^{\pm1}]$ with $\Delta=\Delta(f)$, where $\Delta$ is an $n$-dimensional integral convex polyhedron in $\R^n$ containing the origin. Let $C(\Delta)$ be the cone generated by $\Delta$ in $\R^n.$ 
\begin{Def}\label{defbd}
	The boundary decomposition 
	\begin{equation*}
	B(\Delta)=\{ \text{ the interior of a closed face in }C(\Delta) \text{ containing the origin} \}
	\end{equation*}
	is the unique interior decomposition of $C(\Delta)$ into a disjoint union of relatively open cones. 
\end{Def}
If the origin is a vertex of $\Delta$, then it is the unique 0-dimensional open cone in $B(\Delta)$. Recall that $A_1(f)=(a_{r,s}(f))$ is the infinite semilinear Frobenius matrix whose rows and columns are indexed by the lattice points in $L(\Delta)$. For $\Sigma \in B(\Delta)$, we define $A_1(\Sigma,f)$ to be the submatrix of $A_1(f)$ with $r,s \in \Sigma$. Let $f^{\overline{\Sigma}}$ be the restriction of $f$ to the closure of $\Sigma$. Then $A_1(\Sigma,f^{\overline{\Sigma}})$ denotes the submatrix of $A_1(f^{\overline{\Sigma}})$ with $r,s \in \Sigma$.

Let $B(\Delta)=\{\Sigma_0, \ldots, \Sigma_h\}$ such that $\text{dim}(\Sigma_i)\leq \text{dim}(\Sigma_{i+1})$, $i=0,\ldots,h-1.$ Define $B_{ij}=(a_{r,s}(f))$ with $ r\in \Sigma_i$ and $ s\in \Sigma_j$ $(0\leq i,j \leq h)$. After permutation, the infinite semilinear Frobenius matrix can be written as
\begin{equation}
A_1(f)=
\begin{pmatrix}
B_{00} &  B_{01}  & \cdots\quad &B_{0h}\\
B_{10} &  B_{11}  & \cdots\quad & B_{1h}\\
 \vdots & \vdots & \ddots  & \vdots  \\
B_{h0} & B_{h1}  & \cdots\quad & B_{hh}
\end{pmatrix},
\end{equation} 
where $B_{ij}=0$ for $i>j$. Then $\det(I-TA_1(f))=\prod_{i=0}^h\det(I-TB_{ii})$ and we have the boundary decomposition theorem.
\begin{Th}[Boundary decomposition\cite{Dwan1993}] 
	Let $f\in \F_q[x_1^{\pm1}, \ldots, x_n^{\pm1}]$ with $\Delta=\Delta(f)$. Then we have the following factorization
	\begin{equation*}
	\det(I-TA_1(f))=\prod_{\Sigma \in B(\Delta)}\det \left(I-TA_1(\Sigma,f^{\overline{\Sigma}})\right).
	\end{equation*}
\end{Th}
\begin{Cor}\label{coro12}
	Let $f \in \F_q[x_1^{\pm1}, \ldots, x_n^{\pm1}]$ be a non-degenerate Laurent polynomial with an $n$-dimensional Newton polyhedron $\Delta$. If the origin is a vertex of $\Delta$, then the associated L-function
	\begin{equation*}
	\LF^*(f,T)^{(-1)^{n-1}}=(1-\psi(\Tr(c))T)\prod^{n! \Vol(\Delta(f))-1}_{i=1}(1-\alpha_i T),
	\end{equation*}
	where $c$ is the constant term of $f$, $\Tr: \F_{q} \rightarrow \F_{p}$ is the trace map and $|\alpha_i| \leq q^{n/2}$.
\end{Cor}
\medskip

\subsection{Diagonal local theory}\label{subsec24}
In this subsection, we give some non-degenerate and ordinary criteria for the diagonal Laurent polynomials whose Newton polyhedrons are simplices.
\begin{Def}
	A Laurent polynomial $f \in  \F_{q}[x_1^{\pm1},\ldots, x_n^{\pm1}]$ is called diagonal if $f$ has exactly $n$ non-constant terms and $\Delta(f)$ is an $n$-dimensional simplex in $\R^n.$
\end{Def}
Let $f$ be a Laurent polynomial over  $\F_{q}$,
\begin{equation*}
f(x_1,x_2, \ldots x_n)=\sum_{j=1}^n a_j x^{V_j},
\end{equation*}
 where $a_j \in \F ^*_{q}$ and $V_j=(v_{1j},\ldots, v_{nj})\in \Z^n,\quad  j=1,2,\ldots,n.$ Let $\Delta=\Delta(f)$. Then the vertex matrix of $\Delta$ is defined to be
\begin{equation*}
M(\Delta)=(V_1,\ldots,V_n),
\end{equation*} 
where the $i$-th column is the $i$-th exponent of $f$. If $f$ is diagonal, $M(\Delta)$ is invertible.
 \begin{Prop}\label{propc1}
	Suppose $f\in \F_{q}[x_1^{\pm1},\ldots, x_n^{\pm1}]$ is diagonal with $\Delta=\Delta(f)$. Then $f$ is non-degenerate if and only if $p$ is relatively prime to $\det(M(\Delta))$. 
\end{Prop}
Let $S(\Delta)$ be the solution set of the following linear system
\begin{align*}
M(\Delta)
\begin{pmatrix}
r_1\\ r_2 \\ \vdots \\ r_n
\end{pmatrix}
\equiv 0 (\bmod1),\quad r_i \in \Q \cap [0,1).
\end{align*}
It's easy to prove that $S(\Delta)$ is an abelian group and its order is given by
\begin{align}\label{eq2}
	\left|\det{M(\Delta)}\right|=n!\Vol(\Delta).
\end{align}
By the fundamental structure theorem of finite abelian group, we decompose $S(\Delta)$ into a direct product of invariant factors,
\begin{equation*}
S(\Delta)=\bigoplus_{i=0}^n\Z/d_i\Z,
\end{equation*}
where $d_i|d_{i+1}$ for $i=1,2,\ldots, n-1.$ By the Stickelberger theorem for Gauss sums, we have the following ordinary criterion for a non-degenerate Laurent polynomial\cite{Dwan2004}.
\begin{Prop}\label{prop15}
	Suppose $f\in \F_{q}[x_1^{\pm1},\ldots, x_n^{\pm1}]$ is a non-degenerate diagonal Laurent polynomial with $\Delta=\Delta(f)$. Let $d_n$ be the largest invariant factor of $S(\Delta)$. If $p\equiv 1 (\bmod d_n)$, then $f$ is ordinary at $p$. 
\end{Prop}

\section{Proof of the Main Theorem}\label{sec3}
We prove the main theorem in this section. Let $n=\sum^{r}_{i=1}n_i$ be a partition of a positive integer $n$.
Recall that for $b_{ij} \in \mathbb{Z}_{>0}$ and $p \nmid b_{ij}$ $(1\leq i \leq r,\ 1\leq j\leq n_i)$, the exponential sum $S_k(\vec{a})$ has the expression,
\begin{align*}
S_k(\vec{a})=\sum_{{\sum^{r}_{i=1}\frac{a_i}{\prod_{j=1}^{n_i}x_{ij}^{b_{ij}}}=1}}\psi\left(\Tr_k\left(\sum^{r}_{i=1}\sum^{n_i}_{j=1}x_{ij}\right)\right).
\end{align*}
where $x_{ij}\in\F_{q^k}^*$, $a_i\in \F_{q}^*$, $\psi: \F_{p} \rightarrow \mathbb{C}^*$ is a nontrivial additive character and $\Tr_k: \F_{q^k} \rightarrow \F_{p}$ is the trace map. The L-function associated to $S_k(\vec{a})$ is defined by 
\begin{align*}
\LF(\vec{a},T)=\exp  \left(\sum^\infty _ {k=1} S_k(\vec{a}) \frac{T^k}{k}\right).
\end{align*}
For the sake of brevity, let $N_i=\sum^{i}_{l=1}n_l$ for $1\leq i\leq r$ and $N_0=0$.
Denote $x_{ij}=x_{N_{i-1}+j}$ and $b_{ij}=b_{N_{i-1}+j}$ ($1\leq i \leq r$, $1\leq j \leq n_i$).
Let $f\in \F_{q^k}[x_1^{\pm1},\ldots, x_{n+1}^{\pm1}]$ be the Laurent polynomial defined by  
\begin{align*}
f(x_1,\cdots,x_{n+1})=\sum_{i=1}^{n}x_i+x_{n+1}\left({\sum^{r}_{i=1}\frac{a_i}{\prod_{l=N_{i-1}+1}^{N_i}x_{l}^{b_{l}}}}-1\right).
\end{align*}
We have 
\begin{align}\label{eqs}
S_k(\vec{a})=\frac{1}{q^k}\mathop{\sum_{x_1,\cdots,x_n\in\F^*_{q^k}}}_{x_{n+1}\in\F_{q^k}}\psi\left(\Tr_k\left(f\right)\right)
=\frac{(-1)^n}{q^k}+\frac{1}{q^k}\sum_{x_i\in\F^*_{q^k}}\psi\left(\Tr_k\left(f\right)\right).
\end{align}
Define 
\begin{align*}
S_k^*(f)=\sum_{x_i\in\F^*_{q^k}}\psi(\Tr_k(f))\quad\text{and}\quad \LF^*(f,T)=\exp  \left(\sum^\infty _ {k=1} S^*_k (f) \frac{T^k}{k} \right).
\end{align*}
It follows from formula (\ref{eqs}) that
\begin{align}\label{eql}
\LF(\vec{a},T)^{(-1)^n}=\frac{1}{1-T/q}\LF^*(f,T/q)^{(-1)^n}.
\end{align}

The aim of this paper is to determine the slopes of $\LF(\vec{a},T)^{(-1)^n}$. By formula (\ref{eql}), it suffices to consider the slopes of $\LF^*(f,T)^{(-1)^n}$ instead.
We will prove it later that $\LF(\vec{a},T)^{(-1)^n}$ is a polynomial under a restriction on $p$.
Let $\Delta=\Delta(f)$ denote the Newton polyhedron corresponding to the Laurent polynomial $f$.  Vertices of $\Delta$ are as follows.
\allowdisplaybreaks
\begin{align*}
&V_0 = (0,\cdots,0)\\
&V_1 = (1,0,\cdots,0),\\
&\vdots\\
&V_{n+1} = (0,\cdots,0,1),\\
&V_{n+2} = (\underbrace{-b_1,\cdots,-b_{n_1}}_{n_1},0,\cdots,0,1),\\
&\vdots\\
&V_{n+i+1} = (0,\cdots,0,\underbrace{-b_{N_{i-1}+1},\cdots,-b_{N_i}}_{n_i},0,\cdots,0,1),\\
&\vdots\\
&V_{n+r+1} = (0,\cdots,0,\underbrace{-b_{N_{r-1}+1},\cdots,-b_{N_r}}_{n_r},1).
\end{align*}
	\begin{figure}[htbp] 
	\centering 	
	\begin{tikzpicture}
	\node[right] (0) at (0,-4) {$V_0$};
	\node[left] (1) at (-2,0) {$V_1$};
	\node[above right] (2) at (-1.5,-0.5) {$V_2$};	
	\node[above right] (3) at (0.5,-0.5) {$V_3$};
	\node[right] (4) at (2,0) {$V_4$};
	\node[above] (5) at (0,1.5) {$V_5$};
	\node[above] (6) at (1,0.5) {$V_6$};
	\node[above] (7) at (-0.5,0.5) {$V_7$};
	\fill[fill=blue] (0,-4) circle (2pt);
	\fill (-2,0) circle (1.5pt);
	\fill (-1.5,-0.5) circle (1.5pt);
	\fill (0.5,-0.5) circle (1.5pt);
	\fill (2,0) circle (1.5pt);
	\fill[fill=blue] (0,1.5) circle (2pt);
	\fill (1,0.5) circle (1.5pt);
	\fill (-0.5,0.5) circle (1.5pt);
	\draw (-2,0)--(-1.5,-0.5)--(0.5,-0.5)--(2,0);
	\draw [dashed] (-2,0)--(-0.5,0.5)--(1,0.5)--(2,0);
	\draw (0,-4)-- (-2,0) ;
	\draw (0,-4)-- (-1.5,-0.5) ;
	\draw (0,-4)-- (0.5,-0.5) ;
	\draw (0,-4)-- (2,0) ;
	\draw (0,1.5)-- (-2,0) ;
	\draw (0,1.5)-- (-1.5,-0.5) ;
	\draw (0,1.5)-- (0.5,-0.5) ;
	\draw (0,1.5)-- (2,0) ;
	\draw [dashed] (0,-4)-- (1,0.5) ;
	\draw [dashed] (0,-4)-- (-0.5,0.5) ;
	\draw [dashed] (0,1.5)-- (1,0.5) ;
	\draw [dashed] (0,1.5)-- (-0.5,0.5) ;
	\fill[fill=gray, opacity=0.3] (-2,0)--(-1.5,-0.5)--(0.5,-0.5)--(2,0)--(1,0.5)--(-0.5,0.5);
	\end{tikzpicture}
	\caption{$\Delta$ for $n=4$, $r=2$} 
\end{figure}
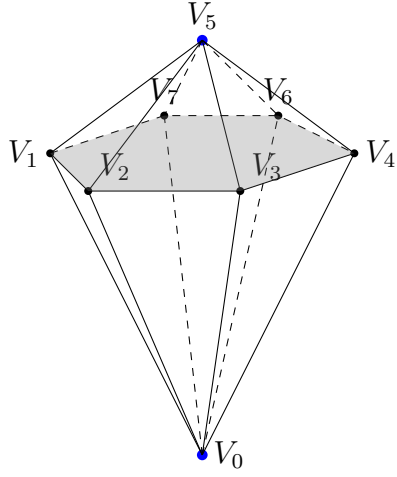
Claim that $\dim\Delta=n+1$ and $\Delta$ has $n+r+2$ vertices. 
Let $\delta_i$ be a co-dimension 1 face of $\Delta$ not containing the origin. Here is a criterion for $\delta_i$.
\begin{Prop}\label{prop20}
	Let $h_i(x)=\sum^{n+1}_{j=1}e_jx_j-1$ be the equation of $\delta_i$, where $e_j$ is uniquely determined rational numbers not all zero.  For any vertex $V$ of $\Delta$, one has $h_i(V)\leq0$.
\end{Prop} 
This deduces the following proposition.
\begin{Prop}
	Every co-dimension 1 face $\delta_i$ contains $V_{n+1}$ as a vertex.
\end{Prop}
Then we can describe some detailed properties for $\delta_i$.
\begin{Prop}\label{eqv}
	Let $\Delta_i$ be the polytope generated by the origin and $\delta_i$.
	Denote $\mathcal{V}=\{V_1,V_2,\cdots,V_n\}$. Every $\Delta_i$ has $n+2$ vertices:
	\begin{align*}
	V_0, V_{n+1},\quad \underbrace{\mathcal{V}-\{V_{i_1}, \cdots, V_{i_k}\}}_{n-k\ \text{vertices}}\quad \text{and} \quad
	\underbrace{\{{V}_{n+1+j_1}, \ldots, V_{n+1+j_k}\}}_{k\ \text{vertices}},
	\end{align*}
	where $0\leq k\leq r$ and $N_{j_l-1}+1 \leq i_l \leq N_{j_l}$, i.e., there is a one-to-one correspondence between vertices $V_{i_l}$ and $V_{n+1+j_l}$ for $1\leq l \leq k$.
	Note that for $l_1\neq l_2$, we have $j_{l_1}\neq j_{l_2}$.
\end{Prop}

\begin{Prop}\label{prop22}
	The equation of $\delta_i$ is given by 
	\begin{align*}
	\sum_{i=1}^{n+1}c_ix_i=1,
	\end{align*}
	where
	\begin{equation*}
	c_i=
	\begin{cases}
	1 \ \ &\text{if} \ \  i \neq i_l,\\
	\displaystyle-\frac{1}{b_{i_l}}\left(\sum_{m=N_{j_l-1}+1}^{N_{j_l}}b_m-b_{i_l} \right) &\text{if} \ \  i = i_l.
	\end{cases}
	\end{equation*}
	Note that  $b_{i_l}$ is the $i_l$-th coordinate of $V_{n+1+j_l}$. 
\end{Prop}
\begin{Prop}\label{prop21}
	There are $\prod^{r}_{i=1}\left(1+n_i\right)$ co-dimension 1 faces of $\Delta$ not containing the origin, i.e.,
	\begin{align*}
	\#\{\delta_i\ \rvert\ \delta_i\ \text{face of}\ \Delta, 0\notin\delta_i\}=\prod^{r}_{i=1}\left(1+n_i\right).
	\end{align*}
\end{Prop}
\begin{proof}
	Recall the selection of vertices in proposition \ref{eqv}. For any fixed $k_0$ ($1\leq k_0\leq r$), we have 
	\begin{align*}
		\#\{\delta_i\ \rvert\ \delta_i\ \text{face of}\ \Delta, 0\notin\delta_i, k=k_0\} = \sum_{i_1\neq i_2\neq\cdots\neq i_{k_0}} n_{i_1}\cdots n_{i_{k_0}}.
	\end{align*}
	Note that $i_i,\cdots,i_k$ lie in $k$ distinct intervals $[N_{j_l-1}+1, N_{j_l}]$. Then we have
	\begin{align*}
	\#\{\delta_i\ \rvert\ \delta_i\ \text{face of}\ \Delta, 0\notin\delta_i\} = 1+\sum_{i=1}^r n_i+\sum_{i\neq j} n_in_j+\ldots + n_1n_2\cdots n_r
	=\prod^{r}_{i=1}\left(1+n_i\right).
	\end{align*}
\end{proof}
Another way to prove proposition \ref{prop21} leads to the following result in combinatorics.
\begin{Cor}
	Let $n$ and $r$ be positive integers. Consider a partition 
	$$n=n_1+n_2+\cdots+n_r$$
	of $n$ as the sum of $r$ integers $n_i$, we have
	\begin{align*}
	\binom{n+r}{n}-\sum^{r}_{i=1}\binom{n+r-n_i-1}{n-n_i-1}=\prod^{r}_{i=1}\left(1+n_i\right).
	\end{align*}
\end{Cor}
\begin{proof}
	To count the number of $\delta_i$, it suffices to consider the combination of $n+1$ linearly independent vertices. Note that all the linearly dependent combinations of $V_j (1\leq j\leq n+r+1)$ contain vertices 
	\begin{align*}
	\underbrace{V_{N_{i-1}+1}, V_{N_{i-1}+2}, \cdots, V_{N_i}}_{n_i}, V_{n+1}, V_{n+i+1},
	\end{align*}
	Since $V_{n+1}$ is a fixed vertex of all $\delta_i$, we have
	\begin{align*}
	\#\{\delta_i\ \rvert\ \delta_i\ \text{face of}\ \Delta, 0\notin\delta_i\} = \binom{n+r}{n}-\sum^{r}_{i=1}\binom{n+r-n_i-1}{n-n_i-1}.
	\end{align*}
\end{proof}
 As a consequence of proposition \ref{prop22}, we have the following result.
\begin{Prop}\label{th5}
	\begin{enumerate}[(i).]
		\item The denominator $D=\mathrm{lcm}\{b_{ij}\ \rvert\ 1\leq i\leq r, 1\leq j\leq n_i\}$.
		\item $f$ is non-degenerate.
		\item $\displaystyle
		\Vol(\Delta)=\frac{1}{(n+1)!}\prod^{r}_{i=1}\left(1+\sum_{j=1}^{n_i}b_{ij}\right)$.
	\end{enumerate}
\end{Prop}
\begin{proof}
	The denominator can be deduced immediately from the equation of $\delta_i$. For each $\delta_i$, the restriction of $f$ to $\delta_i$ is defined by 
	\begin{align*}
	f^{\delta_i}=\sum_{V_j\in\delta_i}a_jx^{V_j}.
	\end{align*}
	Note that each $f^{\delta_i}$ is diagonal. According to proposition \ref{propc1}, we get the condition when $f^{\delta_i}$ is non-degenerate. Using Wan's facial decomposition theorem,  we obtain $(ii)$. 
	
	Now we prove $(iii)$.
	For $1\leq k\leq \prod^{r}_{i=1}\left(1+n_i\right)$, let $\Delta_k$ be the polytope generated by $\delta_k$ and the origin. It can be deduced from the facial decomposition of $\Delta$ \cite{Dwan1993} that
	\begin{align*}
	\Vol(\Delta)=\sum_{k=1}^{\prod^{r}_{i=1}\left(1+n_i\right)}\Vol(\Delta_k).
	\end{align*}
	By proposition \ref{prop22} and formula (\ref{eq2}), we obtain $\Vol(\Delta)$ in this proposition.
\end{proof}
\begin{Prop}
	Let $D$ be the denominator of $\Delta$. The polynomial $f$ is ordinary if $p\equiv 1(\bmod D)$.
\end{Prop}
\begin{proof}
	This proposition follows from proposition \ref{prop15} and Wan's facial decomposition theorem.
\end{proof}

Now we are ready to consider the Hodge number of $\Delta$.
\begin{Th}\label{th6}
	Let $b_{ij}\in\Z_{>0}$ for $1\leq i\leq r$ and $1\leq j\leq n_i$. The number of $H_{\Delta}(k)$ equals the coefficient of $x^k$ in $G(x)$ as follow.
	\begin{align}
	G(x)=\left(1-x^D\right)^{n-r}\prod^{r}_{i=1}\frac{1-x^{\left(1+\sum^{n_i}_{j=1}\frac{1}{b_{ij}}\right)D}}{\prod_{j=1}^{n_i}\left(1-x^{D/b_{ij}}\right)}.
	\end{align}
\end{Th}
\begin{proof}
	By formula (\ref{eqW}) and (\ref{eq3}), to obtain the number of $H_{\Delta}(k)$, it suffices to consider the weight function $w(u)$, where $u$ is a lattice point in $\Z^{n+1}$. Recall that
	\begin{align}\label{eqw}
	w(u)=\inf_{\vec{c}}\left\{\sum^{n+r+1}_{l=1}c_l \ \bigg\rvert\  \sum^{n+r+1}_{l=1}c_lV_l=u, c_l\in\R_{\geq 0}\right\}.
	\end{align}
	By the relationship that $b_{ij}=b_{N_{i-1}+j}$, non-negative integral constant $b_{ij}\ (1\leq i\leq r, 1\leq j\leq n_i)$ can be relabeled as $b_{k}\ (1\leq k\leq n)$. 
	For our example, it can be deduced that $c_l\in\frac{\Z_{\geq 0}}{b_l}$ for $1\leq l\leq n$ and $c_l\in\Z_{\geq 0}$ for $n+1\leq l\leq n+r$. Since $w(u)$ is determined by its associated coefficients,  
	we need to discuss when $\sum^{n+r+1}_{l=1}c_l$ reaches its minimum. 
	
	For $1\leq i\leq r$, divide the first $n$ coordinates of $u$ into $r$ disjoint parts in order, where the $i$-th part has $n_i$ elements. 
	\begin{align*}
	u=(\underbrace{u_1,\cdots,u_{n_1}}_{n_1},\cdots,\underbrace{u_{N_{i-1}+1},\cdots,u_{N_i}}_{n_i},\cdots,\underbrace{u_{N_{r-1}+1},\cdots,u_n}_{n_r},u_{n+1})
	\end{align*}
	Let $M_i=\{V_{N_{i-1}+1},\cdots,V_{N_{i}}\}\cup \{V_{n+1+i}\}$ and $S_i=\{c_{N_{i-1}+1},\cdots,c_{N_{i}}\}\cup \{c_{n+1+i}\}$.  Elements in $S_i$ are coefficients of vertices in $M_i$. Note that $|S_i|=n_i+1$ and $S_{i_1}\cap S_{i_2}=\emptyset$ for $1\leq i_1\leq i_2\leq r$. For our example, formula (\ref{eqw}) implies that each part of coordinates of $u$, i.e., $\{u_{N_{i-1}+1},\cdots,u_{N_i}\}$, are uniquely determined by the vertices in $M_i$ and thus the coefficients in $S_i$. 
	
	We claim that $\sum_{l=1}^{n+r+1}c_l$ reaches its minimum if and only if at least one of elements in $S_i$ is $0$ for all $1\leq i \leq r$. 
	First we prove the necessary condition.
	If we focus on $n_i$ coordinates of $u$, we consider the corresponding $n_i$ coordinates of each vertex in $M_i$. 
	For a vertex $V=(v_{1},v_{2},\cdots,v_{n+1})$, let $$V^{(n_i)}=(\underbrace{v_{N_{i-1}+1},\cdots,v_{N_i}}_{n_i})$$ denote the $i$-th part of coordinates of $V$. Let 
	$$M_i^{(n_i)}=\{V_{N_{i-1}+1}^{(n_i)},\cdots,V_{N_{i}}^{(n_i)}\}\cup \{V_{n+1+i}^{(n_i)}\}.$$
	If we take $V^{(n_i)}$ as a vector in $\Z^{n_i}$, $M_i^{(n_i)}$ is linearly dependent, while any proper subset of $M_i^{(n_i)}$is linearly independent.
	Therefore, if at least one of elements in $S_i$ vanishes, $\sum_{c_l\in S_i}c_l$ is unique and thus minimum. The sum $\sum_{c_l\in S_i}c_l$ reaches minimum for all $1\leq i \leq r$ implies $\sum_{l=1}^{n+r+1}c_l$ reaches minimum.
	The sufficient condition holds as well and can be proved by contradiction. Suppose none of the elements in $S_i$ equal $0$ for some $1\leq i \leq r$. Let $c=\sum_{l=1}^{n+r+1}c_l$. By the selection of vertices in proposition \ref{eqv}, $u$ doesn't lie on any co-dimension 1 face $\delta$ of $c\Delta$, which implies that $\sum_{l=1}^{n+r+1}c_l$ can be reduced.

	It follows that $W_{\Delta}(k)$ equals the number of solutions to 
	\begin{align*}
	\sum^{n+r+1}_{l=1}c_l=k/D,
	\end{align*}
	satisfying that $\prod_{c_l\in S_i}c_l=0$ for $1\leq i\leq r$.
	That is to say, $W_{\Delta}(k)$ equals the number of non-negative integer solutions to 
	\begin{align}\label{eq4}
	\frac{D}{b_1}x_1+\cdots+\frac{D}{b_n}x_n+Dx_{n+1}+\cdots+Dx_{n+r+1}=k,
	\end{align}
	satisfying that $x_{N_{i-1}+1}\cdots x_{N_{i}}x_{n+1+i}=0$ for $1\leq i \leq r$.
	
	Now we determine the generating function for $W_{\Delta}(k)$ with the help of formula (\ref{eq4}).
	\begin{align*}
	g(x)=&\frac{1}{\prod_{k=1}^{n}\left(1-x^{D/b_k}\right)\left(1-x^D\right)^{r+1}}-\sum^{r}_{i=1}\frac{\prod_{j=N_{i-1}+1}^{N_i}x^{D/b_j}x^D}{\prod_{k=1}^{n}\left(1-x^{D/b_k}\right)\left(1-x^D\right)^{r+1}}\\
	&+\mathop{\sum^r_{i_1,i_2=1}}_{i_1<i_2}\frac{\prod_{j=N_{i_1-1}+1}^{N_{i_1}}x^{D/b_j}x^D\prod_{j=N_{i_2-1}+1}^{N_{i_2}}x^{D/b_j}x^D}{\prod_{k=1}^{n}\left(1-x^{D/b_k}\right)\left(1-x^D\right)^{r+1}}-\cdots\\
	&+(-1)^r\frac{\prod_{k=1}^{n}x^{D/b_k}x^{rD}}{\prod_{k=1}^{n}\left(1-x^{D/b_k}\right)\left(1-x^D\right)^{r+1}}\\
	=&\frac{\prod^{r}_{i=1}\left(1-x^{\left(1+\sum^{N_i}_{j=N_{i-1}+1}\frac{1}{b_{j}}\right)D}\right)}{\prod_{k=1}^{n}\left(1-x^{D/b_k}\right)\left(1-x^D\right)^{r+1}}.
	\end{align*}	
	Replace $b_k$ with $b_{ij}$, we obtain
	\begin{align*}
	g(x)=\left(1-x^D\right)^{-r-1}\prod^{r}_{i=1}\frac{1-x^{\left(1+\sum^{n_i}_{j=1}\frac{1}{b_{ij}}\right)D}}{\prod_{j=1}^{n_i}\left(1-x^{D/b_{ij}}\right)}.
	\end{align*}
	Note that $W_{\Delta}(k)$ equals the coefficients of $x^k$ in $g(x)$. By formula (\ref{eq3}), the generating function of $H_{\Delta}(k)$ is as follow.
	\begin{align*}
	G(x)&=\left(1-x^D\right)^{n+1}g(x)\\
	&=\left(1-x^D\right)^{n-r}\prod^{r}_{i=1}\frac{1-x^{\left(1+\sum^{n_i}_{j=1}\frac{1}{b_{ij}}\right)D}}{\prod_{j=1}^{n_i}\left(1-x^{D/b_{ij}}\right)}.
	\end{align*}
	The number $H_{\Delta}(k)$ equals the coefficients of $x^k$ in $G(x)$.
\end{proof}

As a corollary, we obtain the number of $H_{\Delta}(k)$ when $b_{ij}=1$.
\begin{Cor}\label{corow}
	If $b_{ij}=1$ for $1\leq i\leq r$ and $1\leq j\leq n_i$, the number of $H_{\Delta}(k)$ equals the coefficients of $x^k$ in $G_1(x)$ as follow.
	\begin{align}\label{eqw2}
	G_1(x)=\prod_{i=1}^{r}\left(1+x+\cdots+x^{n_i}\right).
	\end{align}
	In this case, we have $H_{\Delta}(0)=1$, $H_{\Delta}(1)=r$ and $H_{\Delta}(k)=H_{\Delta}(n-k)$ for $0\leq k\leq n$.
\end{Cor}

Then we obtain the slopes of $\alpha_i$.
\begin{Th}\label{thm312}
	Let $h_k$ be the number of reciprocal roots of $\LF(\vec{a},T)^{(-1)^n}$ with q-adic slope $k/D$, i.e.
	\begin{align*}
	h_k=\#\{\alpha_i \rvert \ord_q\alpha_i=k/D\}.
	\end{align*}
	
	For $b_{ij}\in\Z_{>0}$ with $1\leq i\leq r$ and $1\leq j\leq n_i$, assume $p\equiv 1(\bmod D)$. The number of $h_k$ equals the coefficients of $x^{k+D}$ in $G(x)$
	for $k=0,1,\cdots,nD$ and $h_k=0$ for any rational number $k\notin\{0,1,\cdots,nD\}$.
	
	In particular, for $b_{ij}=1$ with $1\leq i\leq r$ and $1\leq j\leq n_i$, the number of $h_k$ equals the coefficients of $x^{k+1}$ in $G_1(x)$
	for $k<n$, and $h_k=0$ for $k\geq n$.
\end{Th}
\begin{proof}
	When $p\equiv 1(\bmod D)$, Laurent polynomial $f$ is ordinary. In this case, $h_k=H_{\Delta}(k+D)$ by lemma \ref{le4} and the fact that $\ord_q(\alpha_i)=\ord_q(\beta_i)-1$. Combining corollary \ref{corow} and formula (\ref{eq3}), we obtain $h_k$ in this theorem.
\end{proof}

Using Wan's boundary theorem, we can factor L-function as follow. 
\begin{Th}\label{thm35}
	We have
	\begin{align*}	
	\LF(\vec{a},T)^{(-1)^n}&=(1-T)^{r-1}\prod^{d-r}_{i=1}(1-\alpha_iT),\\
	\LF^*(f,T)^{(-1)^n}&=(1-T)(1-qT)^{r-1}\prod^{d-r}_{i=1}(1-\beta_iT),
	\end{align*}
	where  $d=\prod^{r}_{i=1}\left(1+\sum_{j=1}^{n_i}b_{ij}\right)$ and $\alpha_i=\beta_i/q$.
\end{Th}
\begin{proof}
	Denote $D(q^iT)=\mathrm{det}\left(I-q^iTA_a(g)\right)$. 
	Recall the boundary decomposition $B(\Delta)$ defined in definition \ref{defbd}. Let $N(i)$ denote the number of $i$-dimensional face $\Sigma_{i,j}$ of $C(\Delta)$, where $0\leq i\leq \mathrm{dim}\Delta$ and $1\leq j\leq N(i)$. For Newton polyhedron $\Delta=\Delta(f)$, we have $N(0)=1$ and $N(1)=n+r$. For simplicity, we abbreviate $\Sigma_{i,j}$ to $\Sigma_i$ since the following proof with respect to face $\Sigma_{i,j}$ is independent of the choice of $j$. Note that $\Sigma_i$ is an open cone and $\Sigma_i\in B(\Delta)$. Let $\overline{\Sigma}_i$ be the closure of $\Sigma_i$.
	Denote 
	\begin{equation*}
	D_i'(T)=\mathrm{det}\left(I-TA_a(\Sigma_i,f^{\overline{\Sigma}_i})\right)\ \ \ 
	\text{and}\ \ \ 
	D_i(T)=\mathrm{det}\left(I-TA_a(\overline{\Sigma}_i,f^{\overline{\Sigma}_i})\right).
	\end{equation*}	
	It is obvious that the unique $0$-dimensional cone $\overline{\Sigma}_0$ is the origin
	and $D_0(T)=D_0'(T)=1-T$. When $i=1$, each $f^{\overline{\Sigma}_1}$ can be normalized to $x$ by variable substitution. Explicitly,
	\begin{align*}
	\LF^*(f^{\overline{\Sigma}_1},T)
	=\exp \left(\sum^\infty _ {k=1}-\frac{T^k}{k}\right)=1-T.
	\end{align*}
	By formula (\ref{eq25}), we have
	\begin{align*}
	D_1(T) &= \prod_{i=0}^{\infty} \left( \LF^{*}\left(f^{\overline{\Sigma}_1},q^iT\right)\right)^{\binom{i}{i}}
	= (1-T)\left(1-qT\right)\left(1-q^2T\right)\cdots.
	\end{align*} 
	Since the only boundary of $\overline{\Sigma}_1$ is $\overline{\Sigma}_0$, we get $D_1'(T)$ after eliminating $D_0'(T)$, i.e., 
	\begin{align*}
	D_1'(T)=\frac{D_1(T)}{D_0'(T)}=\left(1-qT\right)\left(1-q^2T\right)\prod_{i=3}^{\infty}\left(1-q^iT\right),
	\end{align*}	
	Using Wan's boundary decomposition theorem,
	\begin{align*}
	D(T)=\prod^{n+1}_{i=1}\prod^{N(i)}_{j=1}D_i'(T)
	=(1-T)(1-qT)^{n+r}\cdots,
	\end{align*}
	By theorem \ref{thm2}, $\LF^*(f,T)^{(-1)^n}$ is a polynomial.
	From formula (\ref{eq24}), we obtain
	\begin{align*}
	\LF^*(f,T)^{(-1)^n}
	=\frac{D(T)D(q^2T)^{\binom{n+1}{2}}\cdots}{D(qT)^{n+1}D(q^3T)^{\binom{n+1}{3}}\cdots} =(1-T)(1-qT)^{r-1}\prod^{d-r}_{i=1}(1-\beta_iT),
	\end{align*}
	where  $d=\prod^{r}_{i=1}\left(1+\sum_{j=1}^{n_i}b_{ij}\right)$. Together with formula (\ref{eql}), we get the factorization for $\LF(\vec{a},T)^{(-1)^n}$.
\end{proof}
\begin{Cor}\label{corol}
	Assume $b_i=1$ and $n_i$ is even for $1\leq i\leq n$. We have
	\begin{align*}	
	\LF(\vec{a},T)^{(-1)^n}&= (1-T)^{r-1}(1-qT)^{\frac{r^2-r}{2}}\prod^{d-\frac{r^2+r}{2}}_{i=1}(1-\alpha_iT),\\
	\LF^*(f,T)^{(-1)^n}&= (1-T)(1-qT)^{r-1}(1-q^2T)^{\frac{r^2-r}{2}}\prod^{d-\frac{r^2+r}{2}}_{i=1}(1-\beta_iT),
	\end{align*}
	where  $d=\prod^{r}_{i=1}\left(1+n_i\right)$ and $\alpha_i=\beta_i/q$. Note that $\ord_q\beta_i \geq 1$ and $\beta_i \neq q\ \text{or}\ q^2$ for $1\leq i\leq d-\frac{r^2+r}{2}$.
\end{Cor}
\begin{proof}
	Proof of this corollary follows from the proof of theorem \ref{thm35}.
	Similarly, when $i=2$, we have $\LF^*(f^{\overline{\Sigma}_2},T)=\frac{1}{1-T}$. Since the boundary of $\overline{\Sigma}_2$ consists of the origin and two sides, we obtain
	\begin{align*}
	D_2'(T)=\left(1-q^2T\right)\prod_{i=3}^{\infty}\left(1-q^iT\right)^{i-1}.
	\end{align*}
	By proposition \ref{prop c}, we have
	\begin{eqnarray}
	D(T)=\prod^{\infty}_{i=1}\left(1-\gamma_iT\right),
	\end{eqnarray}
	where $\#\{\gamma_i\ \rvert\   \mathrm{ord}_q\gamma_i=k\}=W_{\Delta}(k)$ for $k\in \mathbb{Z}_{\geq0}$. Note that $W_{\Delta}(1)=n+r+1$.
	Using Wan's boundary decomposition theorem and Dwork's trace formula, we have
	\begin{align*}
	D(T)
	&=(1-T)(1-qT)^{n+r}(1-q^2T)^{\binom{n+r}{2}+n+r}(1-\gamma_1T)h_1(T),\\
	\LF^*(f,T)^{(-1)^n}&= (1-T)(1-qT)^{r-1}(1-q^2T)^{\frac{r^2-r}{2}}\frac{1-\gamma_1T}{(1-\gamma_1qT)^{n+1}}h_2(T),
	\end{align*}
	where $|\gamma_1|_q=q^{-1}$ and the slopes of reciprocal roots of $h_2(T)$ are greater than 1. 
	
	Now we claim that $\gamma_1$ is non-real and $(1-q^2T) \nmid h_1(T)$. Let $\beta_i$ be a reciprocal root of $\LF^*(f,T)^{(-1)^n}$.
	Under the assumption that $b_i=1$ and $n_i$ is even for $1\leq i\leq n$,  $f$ is an odd function. Then the conjugate $\overline{\beta_i}$ is also a reciprocal root of $\LF^*(f,T)^{(-1)^n}$.
	Theorem \ref{thm2} implies that if $\ord_q\beta_i>(n+1)/2$, then $\beta_i$ is non-real and 
	\begin{align}\label{eq32}
	\ord_q\beta_i+\ord_q\overline{\beta_i}=\omega_i\in\Z\cap[0,n+1].
	\end{align}
	By theorem \ref{corow}, the Hodge number $H_{\Delta}(0)=H_{\Delta}(n)=1$, $H_{\Delta}(1)=H_{\Delta}(n-1)=r$ and $H_{\Delta}(2)=r+\frac{r^2-r}{2}$ since $n_i$ is even. Restricted by formula (\ref{eq32}), 
	$\gamma_1$ is non-real and 
	$(1-q^2T) \nmid h_1(T)$.
\end{proof}


\nocite{*}
\bibliographystyle{amsalpha}
\bibliography{ref}


\end{document}